\documentclass[12pt,letterpaper]{article}
\usepackage{amsmath}
\usepackage{amsfonts}
\usepackage{amsthm}
\usepackage{amssymb}
\usepackage{array}
\usepackage{caption}
\usepackage{color}
\usepackage{enumitem}
\usepackage{epic,eepic}
\usepackage{graphics}
\usepackage{graphicx}
\usepackage{hyperref}
\usepackage{mathtools}
\usepackage{rotating}
\usepackage{verbatim}
\usepackage{mathtools}

\DeclarePairedDelimiter{\ceil}{\lceil}{\rceil}
\DeclarePairedDelimiter{\floor}{\lfloor}{\rfloor}

\newtheorem{prop}{Proposition}
\newtheorem{cor}[prop]{Corollary}
\newtheorem{theorem}[prop]{Theorem}
\newtheorem{lemma}[prop]{Lemma}

\DeclareMathOperator{\CM}{CM}

\begin{document}
\title{On The Cookie Monster Problem}
\author{Leigh Marie Braswell\\Phillips Exeter Academy \and Tanya Khovanova\\MIT}
\maketitle

\begin{abstract}
The Cookie Monster Problem supposes that the Cookie Monster wants to empty a set $S$ of jars filled with various numbers of cookies. On each of his moves, he may choose any subset of jars and take the same number of cookies from each of those jars. The \textit{Cookie Monster number} of $S$, $\CM(S)$, is the minimum number of moves the Cookie Monster must use to empty all of the jars. This number depends on the initial distribution of cookies in the jars. We discuss bounds of the Cookie Monster number and explicitly find the Cookie Monster number for jars containing cookies in the Fibonacci, Tribonacci, $n$-nacci, and Super-$n$-Nacci sequences. We also construct sequences of $k$ jars such that their Cookie Monster numbers are asymptotically $rk$, where $r$ is any real number: $0 \leq r \leq 1$. 
\end{abstract}

\section{The Problem}

In 2002, the Cookie Monster appeared in \textit{The Inquisitive Problem Solver} \cite{VGL}. The hungry monster wants to empty a set of jars filled with various numbers of cookies. On each of his moves, he may choose any subset of jars and take the same number of cookies from each of those jars. The \textit{Cookie Monster number} is the minimum number of moves the Cookie Monster must use to empty all of the jars. This number depends on the initial distribution of cookies in the jars.

We consider the set of $k$ jars with $s_1 < s_2 < \ldots <s_k$ cookies as a set $S = \{s_{1}, s_{2}, \dots, s_{k} \}$. We call $s_{1}, s_{2}, \dots, s_{k}$ a $\emph{cookie sequence}$. We assume no two $s$ are identical, because jars with equal numbers of cookies may be treated like the same jar.  A jar that is emptied or reduced to the number of cookies in another jar is called a \textit{discarded} jar.

Suppose the \textit{Cookie Monster number} of $S$, which we denote $\CM(S)$, is $n$. On move $j$ for $j = 1, 2, \dots, n$ the Cookie Monster removes $a_{j}$ cookies from every jar that belongs to some subset of the jars. We call each $a_{j}$ a \textit{move amount} of the monster. The set of move amounts is represented by the set $A$. Each jar can be represented as the sum of a subset of $A$, or as a sum of move amounts.  This leads to the following observation: 

\begin{lemma}\label{thm:commutativity}
The Cookie Monster may perform his optimal moves in any order.
\end{lemma}

\begin{proof} 
Suppose $\CM(S)= n$ and the Cookie Monster follows an optimal procedure to empty all of the cookie jars in $S$ with $n$ moves. After the Cookie Monster performs move $n$, all of the jars are empty. Therefore, each jar $k \in S$ may be represented as $k = \sum\nolimits_{j \in I_{k}} a_{j}$ for some $I _{k} \subseteq \{1, 2, \dots, n \}$. By the commutative property of addition, the $a_j$ may be summed in any order for each jar $k$. This means that the Cookie Monster may perform his moves in any order he wishes and still empty the set of jars in $n$ moves.
\end{proof}

\section{General Algorithms}

\textit{The Inquisitive Problem Solver} \cite{VGL} presented general algorithms the Cookie Monster may use to empty the jars, none of which are optimal in all cases. These are:
\begin{itemize}
\item The \textit{Empty the Most Jars Algorithm} (EMJA) in which the Cookie Monster reduces the number of distinct jars by as many as possible on each of his moves. For example, if $S = \{7,4,3,1\} $, the monster would remove 4 cookies from the jars containing 7 and 4 cookies or remove 3 cookies from the jars containing 7 and 3 cookies. In both cases, he would reduce the number of distinct jars by 2.
\item The \textit{Take the Most Cookies Algorithm} (TCA) in which the Cookie Monster takes as many cookies as possible on each of his moves. For example, if $S = \{20,19,14,7,4\}$, the monster would remove 14 cookies from the jars containing 20, 19, and 14 cookies. In doing so, he would remove a total of 42 cookies, which is more than if he removed 19 cookies from the two largest jars, 7 cookies from the four largest jars, or 4 cookies from the five largest jars.
\item The \textit{Binary Algorithm} (BA) in which the Cookie Monster removes $2^m$ cookies from all jars that contain at least $2^m$ cookies for $m$ as large as possible on each of his moves. For example, if $S = \{18,16,9,8\} $, the monster would remove $2^4 = 16$ cookies from the jars containing 18 and 16 cookies. 
\end{itemize}

\section{Established Bounds}

There are natural lower and upper bounds for the Cookie Monster number (see O.~Bernardi and T.~Khovanova \cite{BK} and M.~Cavers \cite{C}).

\begin{theorem}\label{thm:Kh1}
Let $S$ be a set of $k$ jars. Then $$\lfloor \log_2k\rfloor + 1 \leq \text{$\CM(S)$} \leq k.$$ 
\end{theorem}

\begin{proof} For the upper bound, the Cookie Monster may always empty the $i$-th jar on his $i$-th move and finish emptying the $k$ jars in $k$ steps. For the lower bound, remember that the $k$ jars contain distinct numbers of cookies, and let $f(k)$ be the number of distinct nonempty jars after the first move of the Cookie Monster. We claim that $k \leq 2f(k)+1$. Indeed, after the first move, there will be at least $k-1$ nonempty jars, but there cannot be three identical nonempty jars. This means that the number of jars plus one cannot decrease faster than twice each time.
\end{proof}

We can show that both bounds are reached for some sets of jars \cite{BK}. 

\begin{lemma}\label{thm:sums}
Let $S$ be a set of $k$ jars with $s_1 < s_2 < \ldots < s_k$ cookies. If $s_i > \Sigma_{k=1}^{i-1}s_k$ for any $i > 1$, then $\CM(S) = k$.
\end{lemma}

\begin{proof}
As the largest jar has more cookies than all the other jars together, any strategy has to include a step in which the Cookie Monster takes cookies from the largest jar only. The Cookie Monster will not jeopardize the strategy if he takes all the cookies from the largest jar on the first move. Applying the induction process, we see that we need at least $n$ moves.
\end{proof}

The lexicographically first sequence that satisfies the lemma is the sequence of powers of 2. The following lemma \cite{C} shows that any sequence with growth smaller than powers of 2 can be emptied in fewer then $k$ moves.

\begin{lemma}\label{thm:Cavers1}
Let $S$ be a set of $k$ jars with $s_1 < s_2 < \ldots < s_k$ cookies. Then $$\CM(S) \leq \lfloor \log_2 s_k \rfloor + 1.$$
\end{lemma}

\begin{proof}
We use the binary algorithm. Choose $m$ such that $2^{m-1} \leq s_k \leq 2^m -1$. Let the set of move amounts $A = \{ 2^{m-1}, \ldots, 2^1, 2^0 \}$. Any integer less than $2^m$ can be expressed as a combination of entries in $A$ by using its binary representation. Because all elements in $S$ are less than $2^m$, all $s_i$ may be written as sums of subsets of $A$. In other words, each jar in $S$ may be emptied by a sum of move amounts in $A$. Therefore, $\text{$\CM(S)$} \leq |A|= m$.  However, $2^{m-1} \leq s_k$ implies that $m \leq 1 + \log_2 s_k$. Therefore, $\text{$\CM(S)$} \leq m \leq 1 + \lfloor \log_2 s_k \rfloor$.
\end{proof}

It follows that the lower bound in Theorem ~\ref{thm:Kh1} is achieved for sets of $k$ jars with the largest jar less than $2^m$. The lexicographically smallest such sequence is the sequence of natural numbers: $S ={1,2,3,4,\ldots,k}$. 

We note that if we multiply the amount of cookies in every jar by the same number, the Cookie Monster number of that set of jars does not change. It follows that the lower bound can also be reached for sequences where all $s_i$ are divisible by $d$ and $s_k < d\cdot 2^m$.

Cavers \cite{C} presented a corollary of Lemma~\ref{thm:Cavers1} that bounds $\CM(S)$ by the diameter of the set $S$.

\begin{cor}\label{thm:Cavers2}
Let $S$ be a set of $k$ jars with $s_1 < s_2 < \ldots < s_k$ cookies. Then $$\CM(S) \leq 2+ \lfloor \log_2 (s_k - s_1) \rfloor .$$
\end{cor}

\begin{proof}
We remove $s_1$ cookies from every jar on the first move.
\end{proof}

It follows that if $S$ forms an arithmetic progression, then $\text{$\CM(S)$} \leq 2 + \lfloor \log_2 (k-1) \rfloor $.

\section{Naccis}\label{sec:nacci}
We now present our Cookie Monster with interesting sequences of cookies in his jars. First, we challenge our monster to empty a set of jars containing cookies in the Fibonacci sequence. The Fibonacci sequence is defined as $F_{0} = 0$, $F_{1} = 1$, and $F_{i} = F_{i-2} + F_{i-1}$ for $i \geq 2$. A jar with 0 cookies and one of the two jars containing 1 cookie are irrelevant, so our smallest jar will contain $F_2$ cookies. Belzner \cite{B} found that the set $S$ of $k$ jars containing $\{ F_2, F_3, \ldots, F_{k+1}\}$ cookies has $\CM(S) = \lfloor \frac{k}{2} \rfloor + 1$. We now present our proof using the well-known inequality \cite{W} for the Fibonacci numbers presented in Lemma~\ref{thm:Fib}.

\begin{lemma}\label{thm:Fib}
For the set of Fibonacci numbers $\{ F_1, F_2, \ldots, F_{k+1}\}$, $$F_{k+1} -1 = \sum\limits_{i=1}^{k-1}F_i.$$
\end{lemma}

\begin{theorem}
For $k$ jars with the set of cookies $S = \{F_{2}, \dots, F_{k+1}\}$, the Cookie Monster number is $\CM(S) = \lfloor \frac{k}{2} \rfloor + 1$.
\end{theorem}

\begin{proof} 
By Lemma~\ref{thm:Fib}, one of the Cookie Monster's moves when emptying the set $S = \{F_{2}, \dots, F_{k+1}\}$ must touch the $k$-th jar containing $F_{k+1}$ cookies and not the first $k-2$ jars containing the set of $\{F_{2}, \dots, F_{k-1}\}$ cookies. The Cookie Monster cannot be less optimal if he starts with this move. Suppose that on his first move, the Cookie Monster takes $F_{k}$ cookies from the $(k-1)$-th and $k$-th jars. This empties the $(k-1)$-th jar and reduces the $k$-th jar to $F_{k+1} - F_k = F_{k-1}$ cookies, which is the same number of cookies that the $(k-2)$-th jar has. Thus, two jars are discarded in one move. He cannot do better than that. By induction, the Cookie Monster may optimally continue in this way. If $k$ is odd, the Cookie Monster needs an extra move to empty the last jar, and the total number of moves is $(k+1)/2$. If $k$ is even, the Cookie Monster needs two extra moves to empty the last two jars, and the total number of moves is $k/2+1$. Therefore, for any $k$, $$\CM(S) = \ceil[\Big] {\frac{k+1}{2}} = \floor[\Big] {\frac{k}{2}} + 1.$$ 
\end{proof}

There exist lesser-known and perhaps more challenging sequences of numbers similar to Fibonacci called $n$-nacci \cite{W2}. We define the $n$-nacci sequence as $N_i = 0$ for $0 \leq i < n - 1$, $N_{i} = 1$ for $n-1 \leq i \leq n$, and $N_{i} = N_{i-n} + N_{i-n+1} + \cdots + N_{i - 1}$ for $i \geq n$. For example, in the 3-nacci sequence, otherwise known as Tribonacci, each term after the third is the sum of the previous three terms. The Tetranacci sequence is made by summing the previous four terms, the Pentanacci by the previous five, and the construction of higher $n$-naccis follows.

We would like to start with the Tribonacci sequence, denoted $T_i$: 0, 0, 1, 1, 2, 4, 7, 13, 24, 44, 81, $\ldots .$ The main property of the Tribonacci sequence, like the Fibonacci sequence, is that the next term is the sum of previous terms. We can use this fact to make a similar strategy for emptying jars with Tribonacci numbers. We will prove that the set $S$ of $k$ jars containing $S = \{T_{3}, \dots, T_{k+2}\}$ cookies has $\CM(S) = \lfloor \frac{2k}{3} \rfloor + 1$ by using inequalities relating the Tribonacci numbers proved in Lemma~\ref{thm:Tri}.

\begin{lemma}\label{thm:Tri}
The Tribonacci sequence satisfies the inequalities:
\begin{itemize}
\item $T_{k+1} > \sum\limits_{i=1}^{k-1}T_i$
\item $T_{k+2} - T_{k+1} > \sum\limits_{i=1}^{k-1}T_i$.
\end{itemize}
\end{lemma}

\begin{proof}
We proceed by induction. For the base case $k = 1$, both inequalities are true. Suppose they are both true for some $k$. If $T_{k}$ is added to both sides of both inequalities' induction hypotheses, we obtain
\begin{eqnarray*}
T_{k+1} + T_k > \sum\limits_{i=1}^{k}T_i
\end{eqnarray*}
\begin{eqnarray*}
T_{k+2} - T_{k+1} + T_k > \sum\limits_{i=1}^{k}T_i.
\end{eqnarray*}
But $T_{k+1} + T_k = T_{k+3} - T_{k+2}$ and $T_{k+2} - T_{k+1} + T_k < T_{k+2}$. Therefore,
\begin{eqnarray*}
T_{k+3} - T_{k+2} = T_{k+1} + T_k > \sum\limits_{i=1}^{k}T_i
\end{eqnarray*}
\begin{eqnarray*}
T_{k+2} > T_{k+2} - T_{k+1} + T_k > \sum\limits_{i=1}^{k}T_i
\end{eqnarray*}
and both inequalities are true for $k+1$. 
\end{proof}

\begin{theorem}
When $k$ jars contain a set of Tribonacci numbers 

\noindent $S = \{T_{3}, \dots, T_{k+2}\}$, then $\CM(S) = \lfloor \frac{2k}{3} \rfloor + 1$.
\end{theorem}

\begin{proof}
Consider the largest three jars. The largest jar and the second largest jar each have more cookies than the remaining $k-3$ jars do in total. That means the Cookie Monster must perform a move that includes the second largest, and possibly the largest jars, and does not touch the smallest $k-3$ jars. Because of the second inequality in Lemma~\ref{thm:Tri}, the largest jar needs to be reduced by one more move that does not touch the smallest $k-3$ jars to be discarded. Thus, because at least two moves are needed to touch and discard the three largest jars, discarding all three jars in two moves is optimal.

This can be achieved if the Cookie Monster takes $T_{k+1}$ cookies from the two largest jars on his first move, emptying the second largest jar and reducing the number of cookies in the largest jar to $T_{k+2} - T_{k+1}$. He then should take $T_{k}$ cookies from the third largest and the largest jars on his second move, emptying the third largest jar and reducing the largest jar to $T_{k+2} - T_{k+1} - T_k = T_{k-1}$ cookies. Now the largest jar is discarded, as the number of cookies left there is the same as the number of cookies in another jar. Because this strategy empties three jars in two moves, the Cookie Monster may optimally continue in this way.

If $k$ has remainder 1 modulo 3, he needs one more move for the last jar bringing the total to $2\lfloor k/3 \rfloor + 1$. If $k$ has remainder 2 modulo 3, he needs two more moves, and the total is $2\lfloor k/3 \rfloor + 2$. If $k$ has remainder 0 modulo 3, he needs three moves for the last group of three bringing the total to $2\lfloor k/3 \rfloor + 1$. Therefore, for any $k$,
$$\text{$\CM(S)$} = \floor[\Big] {\frac{2k}{3} } + 1.$$ 
\end{proof}

More generally, here is the Cookie Monster's strategy for dealing with $n$-nacci sequences, which we call \textit{cookie-monster-knows-addition}. He takes $n - 1$ moves to empty $n$ largest jars. On the $i$-th move he takes $N_{k-i}$ cookies from the $(k-i)$-th largest jar and the largest jar. This way he empties the $(k-i)$-th largest jar and reduces the $k$-th largest jar for each $i$ such that $0 < i < n $. In doing this, $n$ jars are emptied in $n - 1$ moves. This process can be repeated, until at most $n$ elements remain, which he empties one by one.  Thus, when $S = \{N_{n}, \dots, N_{n+k-1}\}$, we will prove that the Cookie Monster number $$\CM(S) = \floor[\Big] {\frac{(n-1)k}{n}} + 1.$$ We first prove the necessary inequalities relating the $n$-nacci numbers. 

\begin{lemma}\label{thm:nac1}
The $n$-nacci sequence satisfies the inequality:
$$N_{k+1} > \sum\limits_{i=1}^{k-1}N_i.$$
\end{lemma}

\begin{proof}
The proof is the same as in Lemma~\ref{thm:Fib} for Fibonacci numbers.
\end{proof}

This inequality is not surprising when we notice that the $a$-nacci sequence grows faster than the $b$-nacci sequence for $a>b$. Hence, other $n$-nacci sequences grow faster than the Fibonacci sequence. But the following inequality is more subtle.

\begin{lemma}\label{thm:nac2}
The $n$-nacci sequence satisfies the inequality:
$$N_{k+n-1} - \sum\limits_{i=k+1}^{k+n-2}N_i > \sum\limits_{i=1}^{k-1}N_i.$$
\end{lemma}

\begin{proof}
By the definition $N_{k+n-1} - \sum\limits_{i=k+1}^{k+n-2}N_i = N_k + N_{k-1}.$ By Lemma~\ref{thm:nac1}, $N_k > \sum\limits_{i=1}^{k-2}N_i$. Hence, $N_k + N_{k-1} > \sum\limits_{i=1}^{k-1}N_i$.
\end{proof}

The next theorem shows that $n$-nacci numbers satisfy many inequalities between the two presented above.

\begin{theorem}\label{thm:nacAll}
The $n$-nacci sequence satisfies the inequality for any $0 \leq j \leq n-2$:
$$N_{k+j} - \sum\limits_{i=k+1}^{k+j-1}N_i > \sum\limits_{i=1}^{k-1}N_i.$$
\end{theorem}

\begin{proof}
By the definition, $N_{k+j} - \sum\limits_{i=k+1}^{k+j-1}N_i = \sum\limits_{i=k+j-n}^{k}N_i$. By the inequality in Lemma~\ref{thm:nac1},

$$ \sum\limits_{i=k+j-n}^{k}N_i = N_k + \sum\limits_{i=k+j-n}^{k-1}N_i >\sum\limits_{i=1}^{k-2}N_i +  N_{k-1} = \sum\limits_{i=1}^{k-1}N_i .$$

\end{proof}

Now we are equipped to come back to the Cookie Monster number.

\begin{theorem}
When $k$ jars contain a set of $n$-nacci numbers $S = \{N_{n}, \dots, N_{n+k-1}\}$, the Cookie Monster number is: $$\CM(S) = \floor[\Big] {\frac{(n-1)k}{n}} + 1.$$
\end{theorem}

\begin{proof}
Consider the largest $n$ jars. The largest $n-1$ jars each have more cookies than the remaining $k-n$ jars do in total. That means the Cookie Monster must perform a move that includes the largest $n-1$ jars and does not touch the smallest $k-n$ jars. Suppose he touches the $(n-1)$-th largest jar on his first move. After that, even if he took cookies from the largest $n-2$ jars during his previous move, the $(n-2)$-th largest jar will still have more cookies than all of the smallest $k-n$ jars combined (due to inequalities in Theorem~\ref{thm:nacAll}). That means there must be another move that touches the $(n-2)$-th largest jar and does not touch the smallest $k-n$ jars nor touches the $(n-1)$-th largest jar. Continuing this, there should be a move that touches the $(n-3)$-th largest jar and does not touch the smallest $k-n$ jars, nor it touches the $(n-1)$-th largest jar nor $(n-2)$-th largest jar, and so on.

Summing up for every jar among the $n-1$ largest jars, there is a move that touches it and possibly the jars larger than it as well as the $n$-th largest jar, but does not touch anything else. Hence, there must be at least $n-1$ moves that do not touch the smallest $k-n$ jars.

We know that we can empty the largest $n$ jars in $n-1$ moves if the Cookie Monster uses his \textit{cookie-monster-knows-addition} strategy.  Thus, because at least $n-1$ moves are needed to touch and discard the last $n$ jars, discarding all $n$ jars in $n-1$ moves is optimal.

We can continue doing this until we have no more than $n$ jars left. Because the smallest $n$ jars in set $S$ are powers of two, we must use the same number of moves as the number of jars left to empty the remaining jars. If $k$ has nonzero remainder $x$ modulo $n$, the Cookie Monster needs $x$ additional moves for the last jars. Hence, the total number of moves is $(n-1) \lfloor k/n \rfloor + x$. If $k$ has remainder 0 modulo $n$, he needs $n$ additional moves to empty the final $n$ jars for a total of $(n-1)\lfloor k/n \rfloor + 1$ moves. Therefore, for any $k$, we save one move for every group of $n$ jars besides the last $n$ jars. Hence, we save $\floor[\Big] {\frac{k-1}{n} }$ moves, and the Cookie Monster number of $S$ is:
$$\CM(S) = k - \floor[\Big] {\frac{k-1}{n} } =  \floor[\Big] {\frac{(n-1)k}{n} } + 1.$$ 
\end{proof}

\section{Super Naccis}\label{sec:supernacci}

The monster wonders if he can extend his knowledge of nacci sequences to non-nacci ones. He first considers Super-$n$-nacci sequences that grow at least as fast as $n$-nacci sequences. Define a Super-$n$-nacci sequence as $S = \{M_{1}, \dots, M_{k}\}$, where $M_{i} \geq M_{i-n} + M_{i-n+1} + \cdots + M_{i-1}$ for $i \geq n$. The Cookie Monster suspects that since he already knows how to consume the nacci sequences, he might be able to bound $\CM(S)$ for Super naccis.

\begin{theorem}
For Super-$n$-nacci sequences $S$ with $k$ terms, $$\floor[\Big] {\frac{(n-1)k}{n} } + 1 \leq \CM(S).$$
\end{theorem}

\begin{proof}
The proof of the bound for $n$-nacci sequences used only the inequalities. The same proof works here.
\end{proof}

\section{Beyond Naccis}\label{sec:beyondnacci}

We found sequences representing $k$ jars such that their Cookie Monster numbers are asymptotically $rk$, where $r$ is a rational number of the form $(n-1)/n$. Is it possible to invent other sequences whose Cookie Monster numbers are asymptotically $rk$, where $r$ is any rational number not exceeding 1?

Before discussing sequences and their asymptotic behavior, we go back to the bounds on the Cookie Monster number of a set and check if any value between the bounds is achieved. 

But first, a definition. A set $S = \{s_1, s_2, \ldots, s_k\}$ of increasing numbers $s_i$ is called \textit{two-powerful} if it contains all the powers of 2 not exceeding $\max (S) = s_k$. We can calculate the Cookie Monster number of a two-powerful set:

\begin{lemma}\label{thm:two-powerful}
Given a two-powerful set $S=\{s_1, s_2, \ldots, s_k\}$, its Cookie Monster number is the smallest $m$ such that $2^m$ is larger than all elements in $S$: $\CM(S) = \lfloor \log_2 s_k \rfloor +1$.
\end{lemma}

\begin{proof}
Let $m$ be the smallest power of 2 not in $S$: $m = \lfloor \log_2 s_k \rfloor+1$. Then $S$ contains a subset of powers of 2, namely $S' = \{2^0, 2^1, \ldots, 2^{m-1}\}$. This subset has a Cookie Monster number $m$. A superset of $S'$ cannot have a smaller Cookie Monster number, so $\CM(S) \geq m$. On the other hand, by Lemma~\ref{thm:Cavers1} $\CM(S) \leq m$. Hence, $\CM(S) = m$.
\end{proof}

Two-powerful sets are important because they are easy to construct, and we know their Cookie Monster number. They become crucial building blocks in the following theorem.

\begin{theorem}\label{thm:exist}
For any $k$ and $m$ such that $m \leq k < 2^m$, there exist a set $S$ of jars of length $k$ such that $\CM(S) = m$. 
\end{theorem}

\begin{proof}
 The given constraint allows us to build a two-powerful set $S$ of length $k$ such that $2^{m-1} \leq s_k < 2^m$. We include in this set all powers of two from 1 to $2^{m-1}$ and any other $k-m$ numbers not exceeding $2^m$. This two-powerful set satisfies the condition.
\end{proof}

Now we return to sequences. Suppose $s_1, s_2, \ldots$ is an infinite increasing sequence. Let us denote the set of first $k$ elements of this sequence as $S_k$. We are interested in the ratio of $\CM(S_k)/k$ and its asymptotic behaviour.

If $s_i = 2^{i-1}$, then $\CM(S_k)/k = 1$. If $s_i = i$, then $\CM(S_k)/k = (\lfloor \log_2 k \rfloor +1)/k$, which tends to zero when $i$ tends to infinity. We know that for Fibonacci numbers the ratio is $1/2$, for Tribonacci it is $2/3$, and for $n$-naccis it is $(n-1)/n$. What about other ratios? Are they possible?

Yes, we claim that any ratio $r: 0 \leq r \leq 1$ is possible. We will prove this by constructing  sequences with any given $r$. The idea is to take a sequence that contains all the powers of 2 and to add some numbers to the sequence as needed. Let us first construct the sequence explicitly.

\subsection{The sequence}\label{sec:sequence}

We build the sequence by induction. We start with $s_1=1$. Then $\text{\text{CM}}(S_1)/1 = 1 \geq r$. We process natural numbers one by one and decide whether to add a number to the sequence by the following rules:

\begin{itemize}
\item If it is a power of 2 we always add it.
\item If it is not a power of 2 we add it if the resulting ratio does not go below $r$.
\end{itemize}

Now we would like to study the sequence and prove some lemmas regarding it. Let us denote the elements of this sequence by $s_i$, its partial sums by $S_k = \{s_1, s_2, \ldots, s_k\}$, and the ratio $CM(S_k)/k$, by $r_k$. By the construction, $r_k \geq r$. We need to prove that $\lim_{k \to \infty} r_k = r$.

Suppose $CM(S_k) = m$ so that the current ratio $r_k$ is $m/k$. If $s_{k+1}$ is a power of two, then $r_{k+1}=(m+1)/(k+1)$ and the difference $r_{k+1}-r_k = (k-m)/k(k+1)$. As $0 \leq k-m < k$, we get  $0 \leq r_{k+1}-r_k \leq 1/(k+1)$. In this case the ratio does not decrease, but the increases are guaranteed to be smaller and smaller as $k$ grows. If $s_{k+1}$ is not a power of two, then $r_{k+1}=m/(k+1)$ and the difference  $r_{k+1}-r_k = -m/k(k+1)$. In this case the ratio always decreases.

\begin{lemma}
If $r=1$, then the sequence contains only powers of 2. If $r=0$, then the sequence contains all the natural numbers.
\end{lemma}

\begin{proof}
We start with the ratio 1 for the first term of the sequence: $r_1 = 1$. Every non-power of 2 in the sequence decreases the ratio. So if $r=1$, we cannot include non-powers of 2. If $r=0$, the ratio $r_k$ is always positive, so we include every non-power of 2.
\end{proof}

The sequences in the previous lemma produce the ratios 0 and 1, so from now on we can assume that $0 < r < 1$. Let us see what happens if we include all numbers between two consecutive powers of 2 in the sequence. Because all powers of 2 are present in the sequence, let us denote the index of $2^m$ in the sequence by $k_m$: $s_{k_m} = 2^m$. Hence, $CM(S_k) = m$ if $k_{m-1} \leq k < k_m$. Also, $r_{k_m} = (m+1)/k_m$. 

\begin{lemma}
If we include all the non-powers of 2 in the sequence between $k_m$ and $k_{m+1}$, then the ratio of ratios is bounded: $r_{k_{m+1}}/r_{k_m} \leq (m+2)/2(m+1)$.
\end{lemma}

\begin{proof}
Suppose by the algorithm we need to add all the numbers between $k_m$ and $k_{m+1}$ to the sequence. Therefore, $k_{m+1} = k_m + 2^m$. The ratios are then $r_{k_m}=(m+1)/k_m$ and $r_{k_{m+1}}=(m+2)/(k_m+2^m)$. So the ratio of ratios is $r_{k_{m+1}}/r_{k_m} =(m+2)/(m+1) \cdot k_m/(k_m+2^m)$. Using the fact that $k_m \leq 2^m$, we get $r_{k_{m+1}}/r_{k_m} \leq (m+2)/2(m+1)$. So as $m$ grows, the ratio is almost halved. Starting from $m=3,$ we can guarantee that this ratio is never more than $2/3$.
\end{proof}

\begin{cor}
If we include all of the non-powers of 2 for $m > 2$, then the ratio of ratios is bounded uniformly: $r_{k_{m+1}}/r_{k_m} < 2/3$.
\end{cor}

\subsection{The theorem}

Now we are ready to prove the theorem.

\begin{theorem}
For any real number $r: 0 \leq r \leq 1$, there exists a sequence $s_i$ with partial sums $S_k = \{s_1, s_2, \ldots, s_k\}$ that have Cookie Monster numbers such that $\text{\text{CM}}(S_k)/k$ tends to $r$ when $k$ tends to infinity.
\end{theorem}

\begin{proof}
As we mentioned before, we can assume that $0< r < 1$. 

Given $r,$ we build the sequence described in Subsection~\ref{sec:sequence}. 
While building the sequence, if we need to skip the next number, we have approached $r$ within $m/k(k+1)$. That is, $$r \leq r_k \leq r+ \frac{m}{k(k+1)} \leq r+ \frac{1}{k+1}.$$

If our sequence contains all but a finite amount of natural numbers, the partial ratio will tend to zero. Because the ratio should never go below $r$, we get a contradiction. Hence, we must drop infinitely many numbers. Each time we drop a number, the partial ratio gets within $1/(k+1)$ of $r$. Therefore, with each next number dropped we get closer and closer to $r$. Now we must prove that not only we can get as close to $r$ as we want, but also that we do not wonder off too far from it in between. 

Take $\epsilon$ such that $\epsilon < r/6$, and consider $k$ such that $1/(1+k) < \epsilon$. We can find a number $t$ such that $t > k$ and $r_t < r + \epsilon$.  Thus, we have approached $r$ within the distance of $\epsilon$, and we continue building the sequence. If the next number is a non-power of 2, then the ratio approaches $r$ even closer. When we reach the next power of 2, then the ratio increases by no more than $\epsilon$. Therefore, the ratio stays within $2\epsilon$ of $r$, so it will not exceed $4r/3$. 

We claim that after this power of 2 we cannot add all non-powers of 2 until the next power of 2. Indeed, if that were the case, then the ratio would drop to a number below $4r/3 \cdot 2/3 < r$. Therefore, we will have to drop a non-power of 2 from the sequence after the first encountered power of 2. We will then approach the ratio again and get at least $\epsilon$-close to it. Thus, for numbers greater than $t$, the ratio will never be more than $2\epsilon$ away from $r$.
\end{proof}

\section{Bibliography and Acknowledgements}\label{sec:acknowledgements}

Cookie Monster is proud that people study his cookie eating strategy. The problem first appeared in \textit{The Inquisitive Problem Solver} by Vaderlind Guy and Larson \cite{VGL}. The research was continued by Cavers \cite{C}, Bernardi and Khovanova \cite{BK} and Belzner \cite{B}. 

Cookie Monster and the authors are grateful to MIT-PRIMES program for supporting this research.

\end{document}